\documentclass[reqno,oneside,12pt]{amsart}

\usepackage[T1]{fontenc}
\usepackage{times,mathptm}
\usepackage{amssymb,epsfig,verbatim,xypic}

\theoremstyle{plain}

\newtheorem{thm}{Theorem}[section]

\newtheorem{pro}[thm]{Proposition}
\newtheorem{lem}[thm]{Lemma}

\theoremstyle{definition}

\newtheorem{rem}[thm]{Remark}

\newenvironment{thm-A}
{\noindent{\bf Theorem A.--}\it}{\medskip}

\newenvironment{thm-M}
{\noindent{\bf Main Theorem.}\it}{\\}

\newenvironment{thm-AA}
{\noindent{\bf Theorem A'.}\it}{\\}

\newenvironment{thm-B}
{\noindent{\bf Theorem B.--}\it}{\medskip}

\newenvironment{thm-C}
{\noindent{\bf Theorem C.--}\it}{\\}

\newenvironment{thm-BP}
{\noindent{\bf Bell-Poonen Theorem.--}\it}{\\}

\newenvironment{thm-BB}
{\noindent{\bf Theorem B'.}\it}

\def\C{\mathbf{C}}

\def\R{\mathbf{R}}

\def\Z{\mathbf{Z}}

\def\Id{{\mathsf{Id}}}

\def\ii{{\sf{i}}}

\def\bbP{\mathbb{P}}

\def\Aut{{\sf{Aut}}}

\def\PGL{{\sf{PGL}}\,}
\def\PSL{{\sf{PSL}}\,}
\def\GL{{\sf{GL}}\,}

\def\Tr{{\mathrm{tr}}}

\def\SO{{\sf{SO}}\,}

\def\SU{{\sf{SU}}\,}
\def\U{{\sf{U}}\,}

\def\SL{{\sf{SL}}\,}

\def\End{{\sf{End}}\,}

\def\Hyp{{\mathbb{H}}}

\def\Aff{{\sf{Aff}}\,}
\def\SU{{\sf{SU}}\,}

\addtocounter{section}{0}             
\numberwithin{equation}{section}       

\begin{document}

\setlength{\baselineskip}{0.54cm}        
%
%
\title[Free actions on complex threefolds]
{Free actions of large groups on complex threefolds}
\date{2019/2020}
\author{Serge Cantat, Olga Paris-Romaskevich, Junyi Xie}
\address{Univ Rennes, CNRS, IRMAR - UMR 6625, F-35000 Rennes, France}
\email{serge.cantat@univ-rennes1.fr}
\email{ olga@pa-ro.net}
\email{ junyi.xie@univ-rennes1.fr}
%
%

%
%

%
%

\begin{abstract} 
We classify compact K\"ahler threefolds $X$ with a free group of automorphisms acting freely on $X$.
\end{abstract}

\maketitle

Recall that an action of a group $\Gamma$ on a set $X$ is \textbf{free} if the stabilizer of any point 
$x\in X$ is reduced to the trivial subgroup $\{\Id\}\subset \Gamma$. If $M$ is a compact complex
manifold, we denote by $\Aut(M)$ its group of automorphisms, i.e. of holomorphic diffeomorphisms.

\medskip

\begin{thm-A} \label{thm:App1}
Let $M$ be a compact K\"ahler manifold of dimension $\leq 3$. If $\Aut(M)$ 
contains a non-amenable subgroup $\Gamma$ acting freely on $M$, then $\dim(M)=3$ and
$M$ is a compact torus $\C^3/\Lambda$.
\end{thm-A}

Moreover, there are examples of non-abelian free groups acting freely on some tori of dimension $3$. First, we 
look at low dimensional tori, then we construct such examples, and we conclude with a proof of Theorem~A.
In Theorem~A, we assume $M$ to be smooth: we refer to Remark~\ref{rem:singular_case} for an extension to singular varieties. 
Section~\ref{par:SL2} gives  counter-examples when $M$ is not K\"ahler. 

\section{Compact tori} Let $A=\C^n/\Lambda$ be a compact torus of dimension $n$.
Every automorphism $f\colon A\to A$
comes from an affine transformation 
\begin{equation}
{\hat{f}}(z)=L(f)(z)+T(f)
\end{equation}
of $\C^n$, where the translation part is a vector $T(f)\in \C^n$ and the linear part $L(f)\in \GL_n(\C)$
preserves the lattice $\Lambda$. This defines a homomorphism $\Aut(A)\to \GL_n(\C)$, $f\mapsto L(f)$.
The following assertions are equivalent:
\begin{itemize}
\item[(i)] $f$ has no fixed point;
\item[(ii)] the image of the linear transformation $(L-{\Id})$ does not intersect the set $T(f)+\Lambda$.
\end{itemize}
In particular, if $\Gamma$ acts freely on $A$, then $\det(L(f)-\Id)=0$ for every $f$ in $\Gamma$. 

\begin{lem}\label{lem:tori-dim2}
Let $A$ be a complex torus of dimension $\leq 2$. If $\Gamma\subset \Aut(A)$ acts freely on $A$, 
then $\Gamma$ is solvable. In particular, every free subgroup of $\Aut(A)$ 
acting freely on $A$ is cyclic.
\end{lem}

\begin{proof}
If $\dim(A)=1$, then $\Aut(A)$ is solvable. 
Assume $\dim(A)=2$, and write $A=\C^2/\Lambda$. A subgroup $G$ of $\GL_2(\C)$
such that $\det(L-\Id)=0$ for every $L\in G$ is solvable; hence, the equivalence of (i) and (ii)
implies: if $\Gamma$ acts freely 
on $A$, the groups $L(\Gamma)$ and $\Gamma$ are solvable.
\end{proof}

\section{Examples} 

\subsection{Closed, real analytic manifolds}\label{par:example_real_analytic}

The group $\SO_3(\R)$ contains a non-abelian free group $\Gamma\subset \SO_3(\R)$.  
This is well known, since the existence of such a group  
is at the basis of the Banach-Tarsky paradox (see~\cite{delaHarpe:Book}). Now, the action of $\Gamma$ on 
$\SO_3(\R)$ by left translations is free, and going to the universal cover $\SU_2$
of $\SO_3(\R)$, we obtain a free action of a non-abelian free group on a simply connected 
manifold. These actions being real analytic, the first assertion of the following theorem is proved. 

\begin{thm} There are real analytic, free actions of non-abelian free groups on the following
real analytic manifolds:
\begin{enumerate}
\item the simply connected, compact Lie group $\SU_2$;
\item the torus $\R^k/\Z^k$ for any $k\geq 2$.
\end{enumerate}
\end{thm}

For the second assertion, it suffices to provide examples in dimension $k=2$; 
we refer to Remark~\ref{rem:Hopf_reel} for such an example.
Here, we describe an example in dimension $3$ which will work also for abelian threefolds.

\smallskip

Consider the lattice $\Z^3$, together with the standard quadratic form of
signature $(1,2)$: $Q(x,y,z)=x^2-y^2-z^2$. Its group of isometries $\SO_{1,2}(\Z)$ is a lattice in the
Lie group $\SO_{1,2}(\R)$. 

\begin{lem}
The group $\SO_{1,2}(\Z)$ contains a free subgroup $\Gamma$ of
 Schottky type such that the eigenvalues 
of every element $g\in \Gamma\setminus\{ \Id\}$ form a triple of real numbers 
$\lambda_g > 1 > \lambda_g^{-1}$.
\end{lem}

\begin{proof}
Consider the subgroup $G_0$ of $\SO_{1,2}(\R)$ preserving each connected component
of $\{Q(x,y,z)=1\}$. If $g$ is an element of this group, and $g$ has an eigenvalue of
modulus $>1$, then its eigenvalues are $\lambda_g > 1 > \lambda_g^{-1}$ for some
$\lambda_g>1$; equivalently, $g$ is a loxodromic isometry of the hyperbolic
space ${\mathbb{H}}=\{(x,y,z)\vert \; Q(x,y,z)=1 \; {\text{and}} \; x>0\}$. 
Now, take two loxodromic isometries $f$ and $g$ in $G_0\cap \SO_{1,2}(\Z)$ to which the
tennis-table lemma of Fricke and Klein applies (see~\cite{delaHarpe:Book}). Then, the group $\Gamma$ generated by 
sufficiently large iterates of $f$ and $g$ is a Schottky group such that each of its element $g\neq \Id$ is loxodromic.
\end{proof}

Choose such a free group $\Gamma$, of rank $2$, and fix a pair $(a,b)$ of elements generating $\Gamma$. 
In \cite{Drumm-Goldman}, Drumm and Goldman 
find a non-empty open subset $\U$ of $\R^3\times \R^3$ such that for every $(s,t)\in \U$, the
affine transformations 
\begin{equation}
A_s(x,y,z)=a(x,y,z)+s, \quad B_t(x,y,z)=b(x,y,z)+t
\end{equation}
generate a free group acting freely and properly on $\R^3$. The group generated by $A_s$
and $B_t$ is an affine deformation $\Gamma_{s,t}$ of $\Gamma$; given any reduced word $w$ in $a$,
$b$, and their inverses, we get an element $w(A_s,B_t)$ in $\Gamma_{s,t}$, which we can write 
\begin{equation}
w(A_s,B_t)(x,y,z)=w(a,b)(x,y,z)+L_w(a,b)(s)+R_w(a,b)(t) 
\end{equation}
where $L_w(a,b)$ and $R_w(a,b)$ are elements of the algebra generated by $a$ and $b$ in $\End(\R^3)$.

Since $a$ and $b$ are in $\SL_3(\Z)$ the group $\Gamma_{s,t}$ acts on the torus $\R^3/\Z^3$. 
This action is free if, and only if, given any non-trivial reduced word $w$, and any element $(p,q,r)$ of 
the lattice $\Z^3$, the equation 
\begin{equation}\label{eq:DGFP}
(\Id-w(a,b))(x,y,z)+ (p,q,r)=L_w(a,b)(s)+R_w(a,b)(t) 
\end{equation}
has no solution $(x,y,z)\in \R^3$. Fix such a pair $(w,(p,q,r))$.  The question becomes: is the vector $L_w(a,b)(s)+R_w(a,b)(t)$
contained in the affine plane $(\Id-w(a,b))(\R^3)+ (p,q,r)$? We distinguish two cases. If $(\Id-w(a,b))(\R^3)+ (p,q,r)$ is actually
a vector subspace, in which case this subspace coincides with $(\Id-w(a,b))(\R^3)$, then we know from Drumm-Goldman 
result that there is a pair $(s,t)$ for which Equation~\eqref{eq:DGFP} has no solution. If $(\Id-w(a,b))(\R^3)+ (p,q,r)$ 
does not contain the origin, then there is no solution to Equation~\eqref{eq:DGFP} if $(s,t)$ is small enough. Thus, 
the set $W(w,(p,q,r))$
of parameters $(s,t)\in \R^3\times \R^3$ such that Equation~\eqref{eq:DGFP} has no solution is
non-empty; hence, $W(w,(p,q,r))$ is open and dense (as the complement of a proper affine subspace of $\R^3\times \R^3$). 
By Baire theorem,
the intersection of all those open dense subsets  
is non-empty: this precisely means that there are pairs $(s,t)$ such that the free group $\Gamma_{s,t}$ acts
freely on $\R^3/\Z^3$.

\subsection{Abelian threefolds} Consider any lattice $\Lambda_0\subset \C$, for instance the lattice
$\Lambda_0=\Z[\ii]$ with $\ii^2=-1$. Set
\begin{equation}
\Lambda=\Lambda_0\times \Lambda_0\times \Lambda_0 \subset \C^3
\end{equation}
and denote by $N$ the abelian threefold $\C^3/\Lambda$.
Now, copy the argument given in the last section (\S~\ref{par:example_real_analytic}), with the same group $\Gamma_{s,t}$, but
viewed as a subgroup of the affine group $\SL_3(\Z)\ltimes \C^3$, acting on $N=\C^3$. We get

\begin{thm}
Let $\Lambda_0\subset \C$ be a cocompact lattice.
There is a  free action of a non-abelian free group on the abelian threefold $(\C/\Lambda_0)^3$ by holomorphic affine transformations.
\end{thm}

\section{Proof of Theorem~A} 

We now prove Theorem~A. According to \cite{Dinh:Tits, Cantat:Annals}, the group $\Aut(M)$ 
satisfies Tits alternative: {\sl{if $\Gamma\subset \Aut(M)$ does not contain a non-abelian free group, then $\Gamma$ contains a solvable
subgroup of finite index and, in particular, is amenable}}. Thus, we  assume that $\Gamma$ is a non-abelian free group, 
acting freely on 
some compact K\"ahler manifold $M$ of dimension $\leq 3$. We shall use several times the following fact. 

\begin{lem}\label{lem:rest-free}
If the group $\Gamma$ stabilizes a subset $S\subset M$, the 
restriction $f\in \Gamma\to f_{\vert S}$ is an injective morphism, and the action 
of $\Gamma$ on $S$ is free. If the action of a finite index subgroup $\Gamma_0\subset \Gamma$
lifts to an action on a finite cover $M'\to M$, then the action of $\Gamma_0$ on $M'$ is free.\end{lem}

\subsection{Kodaira dimension} If $M$ is a complex manifold, we denote by $K_M$ its canonical bundle.

\begin{lem}\label{lem:kod-inv-Z} Let $M$ be a compact K\"ahler manifold, and let $\Gamma$ be a subgroup of $\Aut(M)$.
If the Kodaira dimension of $M$ is non-negative, either $K_M$ is torsion, or there is a finite index subgroup $\Gamma_1$ of $\Gamma$ 
and a $\Gamma_1$-invariant, proper, non-empty, and irreducible Zariski closed subset $Z\subset M$.
\end{lem}

\begin{proof}
 If the Kodaira dimension of $M$ is non-negative, the Kodaira-Iitaka fibration 
provides a surjective morphism $\pi\colon M\to B$ such that: 
\begin{itemize}
\item[(1)] $\pi$ is $\Aut(M)$-equivariant: there is a homomorphism $\rho\colon f\in \Aut(M)\mapsto f_B\in \Aut(B)$ such that 
$f_B\circ \pi = \pi\circ f$;
\item[(2)] the image $\rho(\Aut(M))\subset \Aut(B)$ is a finite group. (see \cite{Ueno}).
\end{itemize}
Thus, a finite index subgroup of $\Gamma$ fixes individually every fiber of $\pi$. If $\dim(B)\geq 1$, 
we take $Z$ to be an irreducible component of some fiber, and $\Gamma_1$ the finite index subgroup of $\Gamma$
that preserves the fiber, as well as every irreducible component of this fiber. 
If $\dim(B)=0$, the Kodaira dimension of $M$ is $0$, and we can fix an integer $d>0$ such that $H^0(M,K_M^{\otimes d})=\C\Omega$ for some
non-trivial section $\Omega$ of $K_M^{\otimes d}$.  There is a homomorphism $\xi\colon \Gamma\to \C^*$
such that 
$f^*\Omega =\xi(f)\Omega
$ for every $f\in \Gamma$. In particular, the divisor $(\Omega)_0$ is $\Gamma$-invariant. If this divisor is empty, 
 then $K_M^{\otimes d}$ is the trivial bundle. If not, we define $Z$ to be an irreducible component of $(\Omega)_0$. 
 \end{proof}

\subsection{Curves and surfaces} 

\begin{lem}\label{lem:inv-curve}
If a free group  acts freely on a curve, the group is cyclic. If $\dim(M)\leq 2$ and $\Gamma$ is a non-abelian free group acting freely on $M$, 
then $\Gamma$
does not stabilize any proper, non-empty, Zariski closed subset. 
\end{lem}

\begin{proof} Since every automorphism of $\bbP^1(\C)$ has a fixed point, we can assume that the genus of the curve 
is at least $1$, but then its automorphism group is virtually solvable, and any free subgroup is cyclic. The second 
assertion follows from the first one and from Lemma~\ref{lem:rest-free}.
\end{proof}

\begin{lem}\label{lem:surfaces-free}
If a free group $\Gamma$ acts freely on a compact, K\"ahler surface, the group is cyclic. 
\end{lem}

\begin{proof} We argue by contradiction, assuming that $\Gamma$ is not cyclic. 

If $h^{2,0}(M)>0$, Lemmas~\ref{lem:kod-inv-Z} 
and~\ref{lem:inv-curve} imply that $K_M$ is trivial.
If $M$ is  a K3 surface, its Euler characteristic is positive, every homeomorphism 
of $M$ has a periodic point \cite{Fuller}, and we get a contradiction. 
If $M$ is a torus, we get a contradiction from Lemma~\ref{lem:tori-dim2}.  
This exhausts all  surfaces with $K_M$ trivial.

Now, assume that $h^{2,0}(M)=0$. Then, from the holomorphic Lefschetz fixed point formula, we must have $h^{1,0}(M)>0$, since otherwise
every automorphism would have a fixed point.
Consider the Albanese morphism $\alpha\colon M\to A_M$, where $A_M$ is the Albanese torus of $M$, and let
$E$ be the image of $\alpha$. Since $h^{1,0}(M)>0$, we get $\dim(E)\in\{1,2\}$.
By Lemma~\ref{lem:rest-free} and~\ref{lem:inv-curve}, every proper  $\Gamma$-invariant analytic subset of $M$ (resp. of $E$) is empty. Thus, $E$ 
is smooth and $\alpha$ is a submersion. First, assume $\dim(E)=1$. 
Since $E\subset A_M$ can not be the Riemann sphere, its group of automorphisms is virtually solvable; thus, the kernel of the homomorphism
$\Gamma\to \Aut(E)$ is a non-abelian free group, acting freely on the fibers of $\alpha$, contradicting Lemma~\ref{lem:inv-curve}. Thus, $\dim(E)=2$ and $\alpha\colon M\to E$ is a finite cover. This implies $h^{2,0}(M)>0$, and we get a contradiction.
\end{proof}

\begin{rem}\label{rem:singular_case}
Let $M$ be a singular complex projective threefold. Denote by $M^0$ 
the singular locus of $M$, by $M^1$ the singular locus of $M^0$, etc. 
Let $\Gamma$ be a non-abelian free group acting freely on $M$ by automorphisms. Then, $\Gamma$ preserves
the subvarieties $M^i$. Since $\Gamma$ is not cyclic, we deduce from Lemma~\ref{lem:inv-curve}
that none of these spaces has dimension $0$ or $1$. So, the singular locus $M^0$ is a smooth complex projective surface, 
and Lemma~\ref{lem:surfaces-free} provides a contradiction. So, Theorem~A extends to the case of complex projective threefolds.
\end{rem}

\subsection{Dimension $3$}  Assume $\dim(M)=3$, $M$ is not a torus, and a non-abelian free group $\Gamma$ acts freely on $M$.

\smallskip

\noindent {\bf{a. ${\text{kod}}(M)\geq 0$. --}} First, assume that the Kodaira dimension of $M$ is non-negative. It follows from Lemmas~\ref{lem:kod-inv-Z} to~\ref{lem:surfaces-free} that
$K_M$ is torsion. Then, after a finite \'etale cover, $K_M$ is trivial and $M$ is one of the following examples:
\begin{enumerate}
\item a torus of dimension $3$;
\item a (simply connected) Calabi-Yau threefold;
\item a product of an elliptic curve with a K3 surface.
\end{enumerate}

\begin{lem}\label{lem:finite-cover-torus}
If a finite cover of $M$ is a torus, and $\Gamma\subset \Aut(M)$ is a non-abelian free group  acting freely on $M$,
then $M$ is a torus.
\end{lem}

Thus, the first case is excluded, since we assume that $M$ is not a torus. 

\begin{proof}
By assumption, there is a torus $A=\C^3/\Lambda$, and a finite group $G$ acting freely on $A$ such that $M=A/G$. 
By construction, $M$ is a quotient of $\C^3$ by a group of affine transformations $\tilde{G}\subset \Aff(\C^3)$; the group
$\Lambda$ is a finite index subgroup of $\tilde{G}$, and the image of the (linear part) homomorphism $L\colon \tilde{G}\to \GL_3(\C)$ is a finite
group (isomorphic to $G$ since the action of $G$ on $A$ is free).

The group $\Gamma$ lifts to a free group of affine transformations of $\C^3$  permuting the orbits of $\tilde{G}$.
When $f$ is an element of $\Gamma$, we denote by $\hat{f}\colon z\mapsto L(f)z+T(f)$ the corresponding affine transformation. 
The group $L(\Gamma)$ normalizes $L(\tilde{G})$, and a finite index subgroup  $L(\Gamma_0)$ commutes to 
every element in $L(\tilde{G})$. If $G$ is non-trivial, then $L(\tilde{G})$ contains a non-trivial linear transformation $S$,  $S$ is diagonalizable
(because it has finite order), and $L(\Gamma_0)$ preserves its eigenspaces. Since the action of $G$ 
on $A$ is fixed-point free, the eigenspace $E_1$ corresponding to the eigenvalue $1$ has positive dimension, intersects $\Lambda$ 
on a lattice $\Lambda_1=E_1\cap \Lambda$, and both $E_1$ and $\Lambda_1$ are invariant under the action of $L(\Gamma_0)$.
If $S$ had three distinct eigenvalues, the three eigenlines of $S$ would be $L(\Gamma_0)$ invariant, contradicting the fact 
that $\Gamma$ is not virtually solvable. Thus, $S$ has exactly one other eigenvalue $\alpha$, corresponding to an $L(\Gamma_0)$-invariant
eigenspace $E_\alpha$, with $E_1\oplus E_\alpha=\C^3$. 

Assume $\dim(E_1)=1$ and $\dim(E_\alpha)=2$. Since $\GL(E_1)$ is abelian, there is a non-abelian free subgroup $\Gamma_1$ of $\Gamma_0$ such that
$L(\Gamma_2)$ acts trivially on $E_1$. Then, there is a non-abelian free subgroup $\Gamma_2$  of $\Gamma_1$ such that the action of $L(\Gamma_2)$ on $E_\alpha$ is made of matrices with eigenvalues $\neq 1$.
Take a generating pair $f$, $g\in\Gamma_2$; computing the commutator $[\hat{f},\hat{g}]$, we observe that the translation part $T([f,g])$ is contained in 
$E_\alpha$ ; this implies that $[\hat{f},\hat{g}]$ has a fixed point in $\C^3$, contradicting the assumption on $\Gamma$. 

The case $\dim(E_1)=2$
is similar: we just have to permute the role of $E_1$ and $E_\alpha$. Thus, $L(\tilde{G})$ is trivial and $M$ is actually a torus. 
\end{proof}

Assume we are in case (3), with a finite cover $M'$ of $M$ isomorphic to $X\times E$ for some $K3$ surface $X$; 
there is a finite group of automorphisms $F\subset \Aut(M')=\Aut(X\times E)$ acting freely on $M'$ such that $M=M'/F$.
Every automorphism $f$ of $M'$ preserves the Albanese fibration $\alpha'\colon M'\to E$: this gives a homomorphism $F\to \Aut(E)$, 
and we denote by $F_E$ its image. The fibration $\alpha'$ determines a fibration $\alpha\colon M\to E/F_E$, and this fibration is
$\Gamma$-invariant. Two cases may occur. Either $E/F_E$ is a curve of genus $1$, its automorphism group is solvable, and 
we get a contradiction with Lemmas~\ref{lem:surfaces-free} and~\ref{lem:rest-free}. Or $E/F_E$ is a Riemann sphere, the fixed points of the elements of 
$F_E$ correspond to the critical values of $\alpha$, and a finite index subgroup of $\Gamma$ preserves the corresponding fibers. 
Again, Lemma~\ref{lem:surfaces-free} provides a contradiction.

\medskip

We can now assume that we are in case (2), i.e. the universal cover of $M$ is an irreducible  Calabi-Yau threefold. 
Lifting $\Gamma$ to the universal cover, we can assume that $M$ itself is Calabi-Yau. 
The action on cohomology gives rise to a homomorphism $\Gamma\to \GL(H^*(M;\Z))$,  
$f\mapsto f_*=(f^{-1})^*$.  Here is the 
key lemma:

\begin{lem}\label{lem:key-lemma}
Let $M$ be a (simply connected) Calabi-Yau threefold. The action of 
$\Aut(M)$ on the cohomology group $H^3(M;\Z)$ factors through a finite group.

Let $f$ be an automorphism of $M$. If all orbits of  $f$ are infinite then 
\begin{enumerate} 
\item the action of 
$f$ on the cohomology of $M$ is virtually unipotent: there is a positive integer $k$ such that $(f^k)^*$
is unipotent;
\item $h^{2,1}(M)=h^{1,1}(M)$ and the topological Euler characteristic of $M$ is $0$.
\end{enumerate}
\end{lem}

\begin{proof} 
Fix a K\"ahler class $\kappa$ on $M$.
Since $h^{p,q}(M)=0$ when $p+q=1$ or $5$, we deduce that every class $\alpha$ in $H^{2,1}(M)$
is primitive: $\alpha \wedge \kappa=0$ because $H^{3,2}(M)=0$. Thus, the intersection product $\int_M\alpha \wedge\overline{\alpha}$
determines an $\Aut(X)$-invariant, positive definite quadratic form on the vector space $H^{2,1}(M)$. 
On $H^{3,0}(M)$, the product $\omega\mapsto \int_M\omega \wedge {\overline{\omega}}$ is also 
positive definite. As a consequence, the image of $\Aut(M)$ in $H^3(M;\C)$ is contained in a unitary group. Since
it preserves the integral structure $H^3(M;\Z)$, it is contained in a finite group. This implies that
a finite index subgroup of $\Aut(M)$ acts trivially on $H^3(M,\Z)$. 

Now, apply the holomorphic Lefschetz fixed point theorem: if there is no periodic point, the traces of $f^*$  
on $H^{1,1}(M)$ and $H^{2,1}(M)$ satisfy 
\begin{equation}
\Tr((f^n)^*_{1,1})=\Tr((f^n)^*_{2,1}) \quad \; (\forall n \in \Z\setminus\{0\}).
\end{equation}  
Changing $f$ in some positive iterate $g=f^m$, we get $\Tr((g^n)^*_{2,1})=h^{2,1}(M)$ 
for all $n$, and then we deduce $\Tr((g^n)^*_{1,1})=h^{2,1}(M)$  for all $n$. This equality implies that 
$(g^*)_{1,1}$ is unipotent  and $h^{2,1}(M)=h^{1,1}(M)$. Thus, $(f^*)_{1,1}$ is virtually unipotent. 
Since $h^{2,0}(M)=0$, the action of $f^*$ on $H^*(M,\Z)$ is virtually unipotent. 
\end{proof}

On a Calabi-Yau threefold, every holomorphic vector field on $M$ is identically $0$; this implies that
$\Aut(M)^0$ is trivial and that the representation $\Aut(M)\to \GL(H^*(M;\Z))$ has a finite kernel (see~\cite{Lieberman}).
So,  $\Gamma$ being a free subgroup of $\Aut(M)$, the representation $\Gamma\to \GL(H^*(M;\Z))$ is faithfull. 
Since $\Gamma$ acts freely on $M$, Lemma~\ref{lem:key-lemma} implies that all elements $f^*\in \GL(H^*(M;\Z))$, for $f\in\Gamma$, are
virtually unipotent: we conclude that $\Gamma$ is cyclic, because a subgroup of $\GL_m(\C)$
all of whose elements are virtually unipotent is a solvable group up to finite index. We obtain the following.

\begin{lem}
Let $M$ be a compact K\"ahler manifold of dimension $3$ whose  Kodaira dimension is non-negative. 
If there is a non-abelian free subgroup of $\Aut(M)$ acting freely on $M$, then  $M$ is a torus. 
\end{lem}

\noindent {\bf{b. ${\text{kod}}(M)=-\infty$. --}} Thus, in what follows, we assume that the Kodaira dimension of $M$ is $-\infty$. In particular, $h^{3,0}(M)=0$.
We prove that  non-abelian free  groups   can not act freely on $M$. 

Assume that $h^{1,0}(M)=0$, and apply the holomorphic fixed point formula. This gives 
$\Tr((f^n)^*_{2,0})=-1
$ for all elements $f\neq \Id$ of $\Gamma$ and all integers $n\neq 0$. Again, we get a contradiction. Thus, 
$h^{1,0}(M)>0$, and  the Albanese map is a non-trivial morphism
\begin{equation}
\alpha\colon M\to A_M,
\end{equation}
where $A_M$ is the Albanese torus of $M$. This map is equivariant with respect to a homomorphism
$\rho\colon \Aut(M)\to \Aut(A_M)$, meaning that $\alpha \circ f = \rho(f)\circ \alpha$ for every $f\in \Aut(M)$. 
Let $E\subset A_M$ be the image of $\alpha$.
Assuming that the rank of the free group $\Gamma$ is at least $2$, we shall prove successively that 
\begin{itemize}
\item $E$ is smooth and the map $\alpha\colon M\to E$ is a submersion;
\item $E$ has dimension $2$,
\end{itemize}
and then we obtain a contradiction.

\medskip

If $E$ contains a non-empty Zariski closed proper subset $Z$ which is invariant under the action of $\rho(\Gamma)$,
then $\alpha^{-1}(Z)$ is a non-empty, Zariski closed, proper and $\Gamma$-invariant subset of $M$; its dimension
is at most $2$, and we get a contradiction since $\Gamma$ can not act freely on such a subset (Lemmas~\ref{lem:rest-free}, and~\ref{lem:surfaces-free}). Thus, $E$ is smooth
and the critical locus of $\alpha$ is empty, i.e. $\alpha$ is a submersion. 

\medskip

If $E$ is a curve, its genus is $\geq 1$
(because $E$ is contained in the torus $A_M$),  its automorphism group is solvable, and there is a non-abelian 
free group $\Gamma_1\subset \Gamma$ acting trivially on $E$, and freely on every fiber of $\alpha$: again, 
we get a contradiction from Lemma~\ref{lem:surfaces-free}.

If $\dim(E)=3$, then $M$ is a finite cover of $E$ (because $\alpha$ is a submersion). This implies that there is a non-trivial 
holomorphic $3$-form on $M$, contradicting $h^{3,0}(M)=0$. 

\medskip

Thus, we assume $\dim(E)=2$.  Let   $K\subset A_M$ be the (connected) subtorus of maximal dimension such that $E+K=E$: it is
uniquely determined by $E$, and the projection $p(E)$ of $E$
in the quotient torus $A_M/K$ is a manifold of general type and of dimension $\dim(E)-\dim(K)$  (more precisely, the canonical bundle of $p(E)$ is ample,
see~\cite{Debarre:Book-Tori}, \S VII). 
If $K=\{0\}$ is  reduced to a point, then, $E$ is a surface of general type, 
hence $\Aut(E)$ is a finite group and we get a contradiction with Lemma~\ref{lem:inv-curve}.  If $\dim(K)=1$ we get $\dim(p(E))=1$ and the 
morphism $p\circ \alpha\colon M\to p(E)$ is invariant under a finite index subgroup of $\Gamma$. Since $p(E)$ is a curve
of general type, the image of $\Gamma$ in $\Aut(p(E))$ is finite and, again, we get a contradiction with Lemma~\ref{lem:surfaces-free}.

Now, assume $\dim(K)=2$, which means that $E=A_M$ is a $2$-dimensional torus. 
Denote by $(x,y)$ the affine coordinates on $E=\C^2/\Lambda$, where $\Lambda$ is a lattice in $\C^2$. 
Let $\Omega$ be a holomorphic $2$-form on $M$. Since $h^{3,0}(M)=0$, we get 
\begin{equation}
\Omega\wedge \alpha^*(dx)=\Omega\wedge \alpha^*(dy)=0.
\end{equation}
This implies that $\Omega=a\alpha^*(dx\wedge dy)$ for some holomorphic function $a\colon M\to \C$; such 
a function must be a constant, and we conclude that $H^{2,0}(M)=\C\alpha^*(dx\wedge dy)$. In particular, 
a non-abelian free subgroup $\Gamma_1$ of $\Gamma$ acts trivially on $H^{2,0}(M)$.

 The holomorphic 
Lefschetz fixed point formula gives $\Tr((f^*_{1,0})^n)=2$ for every $f\in \Gamma_1$ and every $n$; 
this shows that $f^*_{1,0}$ is unipotent. Hence, $\rho(\Gamma_1)$ is a subgroup of $\Aut(E)$, all of whose elements 
are affine automorphisms of $\C^2/\Lambda$ with a unipotent linear part: such a group is solvable. Thus, a non-abelian
free subgroup of $\Gamma$ acts freely on the fibers of $\alpha$ and this contradicts Lemma~\ref{lem:inv-curve}.

This concludes the proof of Theorem~A.

\section{Two questions}

\subsection{}\label{par:question1} Here is a  well known question in algebraic dynamics (it is related to specific forms of the abundance conjecture, on Calabi-Yau manifolds
for instance):

\smallskip

{\noindent}{\bf{Question.--}}  Does there exist an automorphism $f$ 
of a complex projective manifold $X$ such that all orbits of $f$ are Zariski dense, but $X$ is not an 
abelian variety?

\subsection{}  Michael Herman proved that there is a real analytic diffeomorphism $h$ of a (real analytic) compact
manifold $M$ such that (1) the topological entropy of $f\colon M\to M$ is positive, and (2) 
$f$ is a minimal transformation of $M$, meaning that every orbit of $f$ is dense in $M$ for the
euclidian topology (see~\cite{Herman:1981}).  In \cite{Rees1981}, Mary Rees  constructs similar examples on tori of any dimension $\geq 2$, but her examples are 
only $C^0$-smooth (in dimension $2$, positive entropy forces the existence of periodic orbits for ${\mathcal{C}}^2$-diffeomorphisms \cite{Katok:IHES}). 
We don't know whether such an example exists in the context of holomorphic
diffeomorphisms of compact K\"ahler manifolds, even for Calabi-Yau manifolds. The  case $\dim(X)\leq 3$ can be dealt with 
technics of the minimal model program, as in \cite{Lesieutre:ENS}.

\smallskip

{\noindent}{\bf{Question.--}} Does there exist an automorphism $f$ of a compact K\"ahler manifold $X$ 
with positive topological entropy, such that all orbits of $f$ are dense in $X$ for the euclidean topology ? 

\smallskip

If the question raised in Section~\ref{par:question1} has a positive answer, then this last question has a negative answer for any
complex projective variety $X$ (an automorphism of an abelian variety with positive entropy always has orbits which are not Zariski dense).

\section{Compact complex manifolds}\label{par:SL2}

This section concerns compact complex manifolds which are not K\"ahler.

\subsection{Examples on Hopf surfaces}\label{par:example_Hopf} Consider the surface $Y=(\C^2\setminus \{0\})/\langle \Phi\rangle$ where 
$\Phi$ is a homothetic contraction $\Phi(x,y)=(\alpha x, \alpha y)$ with $\alpha \in \C^*$ of modulus $\vert \alpha \vert <1$.
The group $\GL_2(\C)$ acts on $\C^2\setminus\{ 0\}$, commuting to $\Phi$, so it acts on $Y$ holomorphically. An 
element $\gamma\in \GL_2(\C)$ induces a fixed point free automorphism of $Y$ if and only if the set of eigenvalues of
$\gamma$ does not intersect $\alpha^\Z$. If $\Gamma$ is a free subgroup of $\SU_2(\C)$, this last property is
satisfied by every element $\gamma\in \Gamma\setminus \{ \Id\}$, so $\Gamma$ acts freely on $Y$. Then, by the
theorem of Baire, we get: {\sl{there is a dense $G_\delta$ subset ${\mathcal{F}}$ of $\GL_2(\C)^2$ such that every pair of matrices
in ${\mathcal{F}}$ generates a free subgroup of $\Aut(Y)$ acting freely on $Y$}}.

If we quotient $Y$ by a finite order homothety $\Psi(x,y)=(\xi x, \xi y)$, with $\xi$ some root of unity, we get a secondary
Hopf surface with the same property.

\begin{rem}\label{rem:Hopf_reel} Let $T$ be the quotient of $\R^2\setminus\{0\}$ by the homothety $\Phi(x,y)=(2 x, 2 y)$; this 
is a real analytic surface diffeomorphic to $\R^2/\Z^2$. The group $\SL(2,\Z)$ contains a free subgroup $\Gamma$ such that 
every element $\gamma\in \Gamma\setminus \{\Id\}$ is diagonalisable, with two real eigenvalues  $\lambda(\gamma)>\lambda(\gamma)^{-1}$. None of these eigenvalues is in $2^\Z$, because $\lambda(\gamma)$ is not an integer (it is a quadratic integer).
So, $\Gamma$ acts freely on  $T$ by real analytic diffeomorphisms. 
\end{rem}

\subsection{Surfaces}  Let $S$ be a compact complex surface. Assume that $S$ is not K\"ahler, and that $S$ admits an
automorphism $f\colon S\to S$ with no periodic orbit. According to Kodaira's classification, such a surface has a unique minimal
model $\eta\colon S\to S_0$. If $\eta$ is not an isomorphism, it contracts a finite number of rational curves $E_i\subset S$, 
some positive iterate $f^m$ preserves each $E_i$ and has a fixed point on it, thus our assumption 
implies that $S$ is minimal. 

Let $a(S)$ be the algebraic dimension of $S$, i.e. the  transcendance degree  of the field of meromorphic functions ${\mathcal{M}}(S)$. 
A smooth  compact complex  surface with $a(X)=2$ is projective, so $a(S)\in \{0, 1\}$. 

If $a(S)=1$ there is a genus $1$ fibration $\pi\colon S\to B$, onto some curve $B$, such that 
${\mathcal{M}}(S)=\pi^* {\mathcal{M}}(B)$ (see~\cite{BPVdV}, Proposition VI.4.1 and its proof).
This fibration is $\Aut(S)$-equivariant: there is a homomorphism $\rho\colon \Aut(S)\to \Aut(B)$ such that $\pi\circ g=\rho(g)\circ \pi$ for all $g\in \Aut(S)$. 

If $a(S)=0$, then $S$ is a surface of class ${\mathrm{VII}}_0$, i.e. a minimal surface with Kodaira dimension ${\text{kod}}(S)=-\infty$ and first Betti number $b_1(S)=1$. Since $f$ has no periodic orbit, the topological Euler characteristic of $S$ vanishes (see~\cite{Fuller}); equivalently, $b_2(S)=0$.
According to \cite{Teleman}, $S$ must be a Hopf or a Inoue surface (see also~\cite{Bogomolov, Li-Yau-Zheng}).  


\medskip

\begin{thm-B}
Let $S$ be a smooth connected compact complex surface. There is a  non-abelian free group $\Gamma\subset \Aut(S)$  acting 
freely on $S$ if, and only if $S$ is a Hopf surface obtained as a quotient of $\C^2\setminus\{ 0\}$ by a group of homotheties 
$$ 
(x,y)\mapsto (\alpha^m \xi^n x, \alpha^m \xi^n y)
$$
with $\alpha\in \C^*$ of modulus $<1$ and $\xi$ a root of unity. 
\end{thm-B}

\begin{proof} Section~\ref{par:example_Hopf} gives one implication. We now assume that $\Gamma$ is a 
non-abelian free group acting freely on $S$. By Lemma~\ref{lem:surfaces-free}, $S$ is not  K\"ahler.
As seen above, $S$ is minimal, and there are three cases to consider: $\Aut(S)$ preserves a genus $1$ fibration $\pi\colon S\to B$; 
$S$ is a Hopf surface; or $S$ is an Inoue surface. 

Assume that $\Aut(S)$ preserves a genus $1$ fibration. If $\rho(\Gamma)\subset \Aut(B)$ is virtually solvable, 
we argue as in Lemma~\ref{lem:surfaces-free} and obtain a contradiction. So,  $B$ is isomorphic to $\bbP^1(\C)$ and all orbits of
$\rho(\Gamma)\subset \Aut(\bbP^1(\C))$ are infinite; as a consequence, $\rho$ is a submersion and is isotrivial. From~\cite{BPVdV}, 
Section V.5 and Theorem V.5.4, we know that $S$ is a Hopf surface ($S$ can not be a product since it is not K\"ahler).
So, we only need to prove the following lemma. \end{proof}

\begin{lem}
Let $S$ be a Hopf or Inoue surface. Either $\Aut(S)$ contains a finite index solvable subgroup or $S$ is one of the surfaces 
described in Theorem~B.
\end{lem}

\begin{proof} 
First, assume that $S$ is a Hopf surface. 
A finite cover $S'$ of $S$ is a primary Hopf surface, i.e. a quotient of $\C^2\setminus\{ 0\}$ by a holomorphic contraction 
$\Phi\colon \C^2\to \C^2$ fixing the origin. Namba gave normal forms for $\Phi$ and described the automorphism group of 
primary Hopf surfaces according to this normal form; as a corollary, $\Aut(S')$ and $\Aut(S)$ are solvable groups, 
except if $\Phi$ is (conjugate to) a homothety, in which case $\Aut(S')$ is $\GL_2(\C)/\langle \Phi\rangle$ 
(see \cite{Namba}, as well as \cite[Section I]{Wehler}). This last property is based on the 
following remark. 

\begin{rem} Suppose $\Phi(x,y)=(\alpha x, \alpha y)$ 
with $\alpha$ of infinite order, and consider  an automorphism $g$ of $\C^2\setminus\{0\}$
such that $g\Phi g^{-1}=\Phi^{\pm 1}$. By Hartogs 
theorem, $g$ extends to an automorphism of $\C^2$ fixing the origin. Then the equation $g(\alpha x, \alpha y)=\alpha^{\pm 1} g(x,y)$ implies that
$g\in \GL_2(\C)$ and  $g\Phi g^{-1}=\Phi$. \end{rem}

If $S$ is a secondary Hopf surface, $S=S'/G$ for some finite subgroup $G$ of $\Aut(S')$ acting freely on $S'$. 
The fundamental group of $S$ is 
a central extension $0\to \Z\to \Lambda \to G\to 1$, where $\Z$ corresponds to $\pi_1(S')$; we can also identify $\Lambda$
to the group of deck transformations on the universal cover $\C^2\setminus\{ 0\}$ of $S$, 
with $\Z$ generated by the homothety $\Phi$, and from the last remark we see that $\Lambda$ is a subgroup of $\GL_2(\C)$.
Note that the action of $\Lambda\subset \GL_2(\C)$
on $\bbP(\C^2)=\bbP^1(\C)$ factorizes as an action of $G$; we denote by ${\sf{P}}(G)$ its image. 

Lifting the group $\Aut(S)$ to the universal cover, we obtain 
an extension 
\begin{equation}
1\to \Lambda \to {\tilde \Aut(S)} \to \Aut(S)\to 1.
\end{equation}
The group $\Aut(S)$ normalizes the group $\Lambda$, hence also the subgroup generated by $\Phi^k$ for some sufficiently divisible 
integer $k\geq 1$. The last remark, applied with $\Phi^k$, shows that ${\tilde \Aut(S)}$ is contained in $\GL_2(\C)$.

The subgroup ${\sf{P}}(G)$ is normalized by the image of ${\tilde \Aut(S)}$ in $\PGL_2(\C)$, and a finite index subgroup  of ${\sf{P}}({\tilde \Aut(S)})$
commutes to all elements of ${\sf{P}}(G)$. If ${\sf{P}}(G)$ were not trivial, ${\tilde \Aut(S)}$ would be virtually solvable, because 
the centralizer of a finite order element of $\PGL_2(\C)$ if cyclic or dihedral. So, ${\sf{P}}(G)$ is trivial and $\Lambda$ is 
contained in the group of homotheties $\C^*$.  Since $\Lambda$ is discrete, 
it is generated by $\Phi$ and a homothety $\Psi(x,y)=(\xi x, \xi y)$ of finite order. 
We conclude that $S$ is one of the Hopf surfaces of Theorem~B.


There are three types of Inoue surfaces, but in all cases the universal cover $\tilde{S}$ is isomorphic to 
$\Hyp\times \C$, where $\Hyp\subset \C$ is the upper half plane. By Liouville's theorem, every holomorphic function 
from $\C$ to $\Hyp$ is constant; this implies that every holomorphic diffeomorphism 
$F$ of $\Hyp\times \C$ preserves the fibration onto $\Hyp$:  $F(z,w)=(u_F(z),a_F(z) w+b_F(z))$ 
where $u_F$ is in $\PSL_2(\R)$, acting by homographies on $\Hyp$, $a_F$ and $b_F$ are holomorphic
functions of $z\in \Hyp$, and $a_F$ does not vanish. The group of deck transformations $\Lambda\simeq \pi_1(S)$
of the universal cover $\Hyp\times \C\to S$ fixes a unique point on the boundary of $\Hyp$; putting
this point at infinity, the action of $\Lambda$ on $\Hyp$ is generated by real translations $z\mapsto z+a_i$ and 
a scalar multiplication $z\mapsto \alpha z$ (see \S~2 and~3 of~\cite{Inoue}). Let ${\tilde \Aut(S)}$ be the group of all possible lifts of 
automorphisms of $S$ to its universal cover (an extension of $\Aut(S)$ by $\Lambda$). This group
normalizes the action of $\Lambda$, so every element ${\tilde f}$ of ${\tilde \Aut(S)}$ can be written 
${\tilde f}(z,w)=(u z + b, a(z)w+b(z))$, with $u\in \R^*$ and $v\in \R$. Hence, ${\tilde \Aut(S)}$ 
is a solvable group. 
\end{proof}

\subsection{An example in dimension $3$} Let $\Lambda$ be a lattice in the Lie group $\SL_2(\C)$. The 
quotient space $M=\SL_2(\C)/\Lambda$ is a compact complex manifold of dimension $3$; it is not K\"ahler (for instance
it contains homologous curves of arbitrary large area). 
Let $\gamma$ be an element of $\SL_2(\C)$. Its action by left multiplication $g\mapsto \gamma\cdot g$ on $\SL_2(\C)$ 
induces an automorphism, i.e. a holomorphic diffeomorphism, $L_\gamma\colon M\to M$. The point $g\Lambda\in M$ 
is fixed by $L_\gamma$ if and only if $g^{-1} \gamma g\in \Lambda$. 

\begin{pro}
There is a non-abelian free group $\Gamma\subset \SL_2(\C)$ acting freely on $M$ by left translation. There is a dense
$G_\delta$  subset ${\mathcal{F}}$ in $\SL_2(\C)\times \SL_2(\C)$ such that 
any pair $(\gamma_1,\gamma_2)\in  {\mathcal{F}}$ generates a free subgroup of rank $2$ in $\SL_2(\C)$ 
acting freely on $M$.
\end{pro}

\begin{proof}
Let $\gamma$ be an element of $\SU_2(\C)$. If $g^{-1}  \gamma  g\in \Lambda$ 
its iterates $g^{-1} \gamma^n g$ form a bounded sequence in the discrete group $\Lambda$ so that $\gamma$ has finite order. Thus, any free subgroup of $\SU_2(\C)$ acts freely 
on $M$ by left translation. The second assertion follows from the theorem of Baire. 
\end{proof}

\bibliographystyle{plain}
 
\bibliography{references}

\end{document}